\documentclass[11pt]{amsart}
\usepackage{amssymb, amscd, amsmath, amsthm, epsf,  latexsym,color ,mathtools}

\usepackage[lmargin=1.1in,rmargin=1.1in,tmargin=1in,bmargin=1in]{geometry}
 \usepackage{setspace} 
\usepackage[pdfpagelabels,bookmarksdepth=3]{hyperref}
\hypersetup{colorlinks=true,citecolor=black,urlcolor=black,linkcolor=black}
\usepackage{pinlabel}
\usepackage{pb-diagram}
\usepackage{soul}
 \usepackage{tikz-cd}
\usepackage{comment}
\usepackage{float}
\usepackage{graphicx}
\usepackage{mathrsfs}
\usepackage{wrapfig}
\restylefloat{figure}
\usepackage{caption}
\usepackage{subcaption}
\usepackage{relsize}

 \usepackage{xcolor}

\newtheorem{theoremABC}{Theorem}

\def\url#1{\expandafter\string\csname #1\endcsname}

\numberwithin{equation}{section}
\newtheorem{theorem}{Theorem}[section]
\newtheorem*{theorem*}{Theorem}
\newtheorem{corollary}[theorem]{Corollary}
\newtheorem{lemma}[theorem]{Lemma}

\newtheorem{proposition}[theorem]{Proposition}

\theoremstyle{definition}
\newtheorem{definition}[theorem]{Definition}

\newtheorem{remark}[theorem]{Remark}

\DeclareMathOperator{\Tr}{Tr}

\newcommand{\ep}{\epsilon}

\newcommand{\bbi}{{{\bf i}}}
\newcommand{\bbj}{{{\bf j}}}
\newcommand{\bbk}{{{\bf k}}}

\newcommand{\ZZ}{{\mathbb Z}}
\newcommand{\RR}{{\mathbb R}}

\newcommand{\CC}{{\mathbb C}}

\newcommand{\cL}{{\mathcal L}}

\newcommand{\cM}{{\mathcal M}}

\newcommand{\Rea}{{\operatorname{Re}}}

\newcommand{\RN}[1]{%
  \textup{\uppercase\expandafter{\romannumeral#1}}%
}

\DeclareMathOperator{\Hom}{Hom}

\DeclareMathOperator{\Span}{Span}

\DeclareMathOperator{\ad}{ad}

\graphicspath{ {figures/} }

\title[Tangles, relative character varieties, and 
perturbed flat  moduli spaces] {Tangles, relative character varieties, and holonomy 
perturbed traceless flat  moduli spaces}

\author{Guillem Cazassus}
\address{Mathematical institute,  University of Oxford, Oxford
OX2 6GG}
\email{g.cazassus@gmail.com}

\author{Chris Herald}
\address{Department of Mathematics and Statistics, University of Nevada,  Reno, NV 89557} 
\email{herald@unr.edu}

 \author{Paul Kirk}
\address{Department of Mathematics, Indiana University, Bloomington, IN 47405} 
\email{pkirk@indiana.edu}

\thanks{CH was supported by Simons Collaboration Grants for Mathematicians, GC was funded by EPSRC grant reference EP/T012749/1.}

 \date{Jan.~30, 2021}                                           


\begin{document}

\begin{abstract} We  prove that the restriction map from the subspace of regular points of the holonomy perturbed SU(2) traceless flat moduli space of a tangle in a 3-manifold to the traceless flat moduli space of its boundary marked surface is a Lagrangian immersion. 
 A key ingredient in our proof is the use of composition in the Weinstein category, combined with the fact that SU(2) holonomy  perturbations in a cylinder induce Hamiltonian isotopies.
In addition, we show that $(S^2,4)$, the 2-sphere with four marked points, is its own traceless flat  SU(2) moduli space. 

\end{abstract}
\maketitle

\section{introduction}

We gather together some of the key symplectic properties of character varieties and traceless character varieties, as well as  variants which correspond to perturbed flat moduli spaces that arise in the gauge theoretic study of 3-manifolds.  Some of these results are well known to the experts, but the proofs in the literature are framed  in   contexts that include gauge theory, Hodge theory, and symplectic reduction.  In the present exposition, we provide   a  general proof of the fact that, roughly,  the character variety of a 3-manifold provides an immersed Lagrangian in the character variety of its boundary surface, for any compact Lie group $G$, whether the 3-manifold and its boundaries include tangles, or whether there are trace  or other conjugacy restrictions on some meridional generators, and furthermore we extend  the result to the holonomy perturbed situation.    Moreover, we clarify why different natural definitions of symplectic structures on the pillowcase, arising  as the character variety of the torus or the 2-sphere with four marked points, are equivalent. The  results are proved using only the   Poincar\'e-Lefschetz duality theorem, basic algebraic topology, and the notion of   composition in the Weinstein category.

\bigskip
 
 Let $(X,\cL)$ be a tangle in a compact, oriented 3-manifold $X$; that is, assume that $\cL$ is a properly embedded, compact 1-manifold.   For our initial discussion, we consider $G=SU(2)$.   
 Let $\pi$ denote some holonomy perturbation data supported in a finite disjoint union of solid tori in the interior of $X\setminus \cL$. 
 
 This data determines two moduli spaces, $$\cM_\pi(X,\cL),$$ the moduli space of   {\em $\pi$ holonomy-perturbed flat  $SU(2)$ connections on $X\setminus \cL$ with traceless holonomy on small meridians of $\cL$}, and the (well studied)  moduli space:
 $$\cM(\partial  X,\partial \cL )$$  of {\em flat $SU(2)$ connections on the punctured surface $\partial X\setminus \partial \cL $ with traceless holonomy 
around the marked points $\partial \cL$}.

Holonomy perturbed flat connections  on a 3-manifold are flat near the boundary, and  restriction to the boundary defines a map
\begin{equation}
\label{rest} r: \cM_\pi(X,\cL)\to\cM(\partial X, \partial \cL)
\end{equation}

The moduli space $\cM(\partial  X,\partial \cL )$ is the cartesian product of the moduli spaces of its path components. The flat $SU(2)$ moduli space of an oriented connected surface of genus $g$ with $k$ marked points is known to be a singular variety  with smooth top stratum carrying a symplectic form called the {\em Atiyah-Bott-Goldman form} \cite{Atiyah_Bott,Goldman}.  Thus   the smooth top stratum of the cartesian product $\cM(\partial  X,\partial \cL )$, denoted $\cM(\partial  X,\partial \cL )^*$,   is endowed with the product symplectic form.

 \medskip

The main result of this article is the following theorem (Theorem \ref{thm1A} below), concerning the regular points (see Section \ref{regpts})    of the perturbed moduli space.  

\begin{theoremABC}  \label{thm1}  
Suppose $A\in \cM_\pi(X,\cL)$ is a regular point.
Then $A$ has a neighborhood $U$ so that   the restriction
$r|_U:U\subset \cM_\pi(X,\cL)\to \cM(\partial X, \partial \cL)^*$ is a Lagrangian immersion.
 
\end{theoremABC}

This is not a surprising result; indeed,  many special cases are known, for example, when $\cL$ is empty this result is proven  in \cite{Herald}.  
 Our primary aim
 is to provide details of the assertion that well-known arguments
 in the flat  case   extend to the holonomy-perturbed flat case when $\cL$ is  nonempty,  in support of one claim of the main result of the article \cite{CHKK}.   In that article it is shown that a certain process introduced by Kronheimer-Mrowka to ensure admissibility of bundles in instanton homology \cite{KM}  manifests itself on the  symplectic side of the Atiyah-Floer conjecture  \cite{MR974342}  (i.e., Lagrangian Floer theory of character varieties or flat moduli spaces)  as  a certain Lagrangian immersion of a smooth closed genus 3 surface into the smooth stratum $\mathbb{P}^*\times \mathbb{P}^*$ of
 the product of two pillowcases  (cf. Equation (\ref{pillowcaseeq})). What is proved in \cite{CHKK} is that this genus 3 surface satisfies the hypotheses of Theorem \ref{thm1} at every point.

 Since it causes no extra   work, we take the opportunity to provide	 an elementary algebraic topology proof of Theorem \ref{thm1} {\em in the flat case}, for any compact Lie group $G$. The statement can be found in Corollary \ref{flatthmtangle}.   We emphasize that the flat case of Theorem \ref{thm1}, when $\cL$ is nonempty, is known (e.g., to those who attended the appropriate Oxford seminars in the 1980s) and indeed discussed on pages 15-16 of Atiyah's monograph \cite{MR1078014}.

\medskip

Having stated Theorem \ref{thm1} and described its relation to the gauge theoretic literature using the language of flat and perturbed flat moduli spaces, we note that these spaces can also be identified with certain character varieties, the definitions of which do not require any of the analytical machinery of gauge theory;  the proofs in this article are most simply explained without it, so we shall henceforth revert to the character variety terminology in our exposition.  In the simplest situation of a connected 2- or 3-manifold $M$ with base point $x_0$, for a fixed compact Lie group $G$, each flat $G$ connection determines a holonomy representation from $\pi_1(M,x_0)$ to $G$, and this correspondence induces a bijection   between the  moduli space of flat connections  and the set of $G$ representations of $\pi_1$ modulo conjugation, known as the {\em character variety}.   We describe various extensions of the notion of the character variety corresponding to traceless and perturbed flat moduli spaces below.

  We also  take this opportunity to provide an elementary exposition of   holonomy perturbations   from the perspective of fundamental groups and character varieties, in the language of composition of Lagrangian immersions.
  Algebraic topology arguments simplify the task of  explaining how to understand the extension from the flat  to the holonomy perturbed flat situation algebraically.     Other arguments, which instead appeal to Hodge and elliptic theory of  perturbed  flat bundles over  3-manifolds with boundary, can be made when $\cL$ is empty, and can be found in detail in \cite{Taubes, Herald}.   But when $\cL$ is nonempty, proper treatment of the traceless condition requires more  subtle  analytic tools.

 Our focus on holonomy perturbations  is motivated by the fact that they are  compatible with the  analytical framework of  the instanton gauge theory side  of the Atiyah-Floer conjecture \cite{Taubes,Flo}. Holonomy perturbations modify the flatness condition (i.e. the non-linear PDE {\em Curvature$=0$}) in a specific way on a collection of solid tori,  as described  in Lemma 8.1 of Taubes \cite{Taubes} (see also Lemma 16 of \cite{Herald}).  We translate this result into the language of character varieties in Section \ref{su2ahp}.

\medskip

In addition to proving Theorem \ref{thm1}, we prove Theorem \ref{PC=PC2}, whose statement roughly says {\em the four punctured 2-sphere is its own traceless $SU(2)$ moduli space}.  This is a variant, for $(S^2,4)$, of the 
ubiquitous mathematical statement that a torus is isomorphic to its Jacobian.
We now set some notation in preparation for the formal statement.
First, denote $U(1)\times U(1)$ by $\mathbb{T}$.
The group $SL(2,\ZZ)$ acts on $\mathbb{T}$ via 
\begin{equation}\label{SL2Z}\begin{pmatrix} p&r\\q&s\end{pmatrix}
\cdot (e^{x\bbi},e^{y\bbi})=(e^{(px+ry)\bbi},e^{(qx+sy)\bbi}).\end{equation}
The 2-form $dx\wedge dy$ defines the standard symplectic structure on $\mathbb{T}$ and is invariant under  the $SL(2,\ZZ)$ action.  
Note that the action of the central element $-1\in SL(2,\ZZ)$  on $\mathbb{T}$ defines the {\em elliptic involution}, denoted by
 $ \iota(e^{\bbi x}, e^{\bbi y})=(e^{-\bbi x}, e^{-\bbi y})$;  this involution has    four fixed points $(\pm 1, \pm1)$. Denote by $\mathbb{T}^*\subset \mathbb{T}$ the complement of the four fixed points, on which $\iota$ acts freely.    Then set 
\begin{equation}\label{pillowcaseeq}\mathbb{P}=\mathbb{T}/\iota \text{ and }\mathbb{P}^*=\mathbb{T}^*/\iota.\end{equation}
The quotient $PSL(2,\ZZ)=SL(2,\ZZ)/\{\pm1\}$ acts on the smooth locus  $\mathbb{P}^*$, a 4-punctured 2-sphere, preserving the 
symplectic form $dx\wedge dy$.

\medskip

  \begin{wrapfigure}[14]{r}{0.37\textwidth}
 \vspace{-.9cm}
\labellist 
\pinlabel $a$ at 150 238
\pinlabel $b$ at 200 120
\pinlabel $c$ at 140 110
\pinlabel $d$ at 60 115
\pinlabel $H_2$ at 130 23
\pinlabel $H_1$ at 105 105
\endlabellist
\centering
\includegraphics[width=0.3\textwidth]{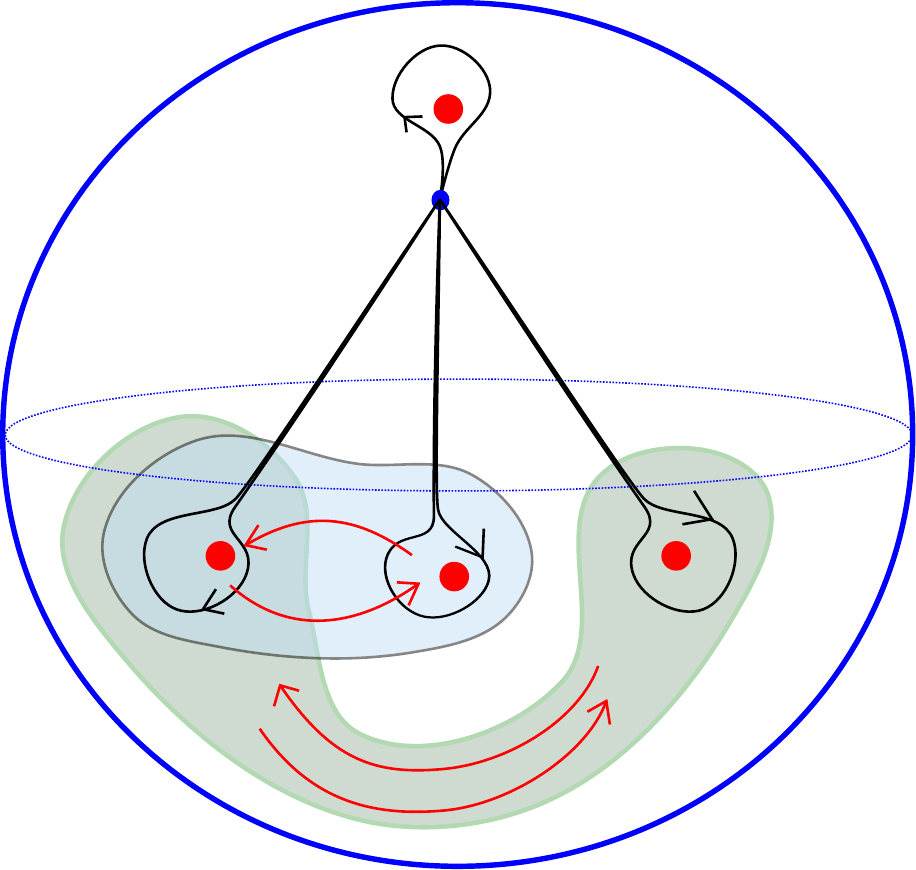}\caption{}
\label{halffig}
\end{wrapfigure}

Next,  we consider the relative character variety (defined below) of the   2-sphere with four marked points $\chi_{SU(2),J}(S^2,4)$, where the $J$ subscript indicates the four meridians are sent to  the conjugacy class $J$ of traceless elements. This character variety is equipped with its relative Atiyah-Bott-Goldman 2-form $\omega_{(S^2,4)}$ (defined below).

 From  the presentation 
$\pi_1(S^2-4)=\langle a,b,c,d~|~ abcd=1\rangle,$ we see that this group is freely generated by $a,b,c$.  Hence a representation is uniquely defined by where it sends the generators $a,b,c$.
Define the function (identifying $SU(2)$ with the group of unit quaternions; see Section \ref{LGG}):
$$\hat\rho:\mathbb{P}\to \chi_{SU(2),J}(S^2,4),~
[e^{\bbi x}, e^{\bbi y}]\mapsto \big[a\mapsto\bbj, b\mapsto e^{\bbi x}\bbj, c\mapsto  e^{\bbi y}\bbj\big]. $$

 \begin{theoremABC}\label{PC=PC2}  Half Dehn twists in the two twice-punctured  disks 
 indicated in Figure \ref{halffig}  generate a $PSL(2,\ZZ)$ action on $\chi(S^2,4)$ which preserves the Atiyah-Bott-Goldman symplectic form $\widehat \omega_{(S^2,4)}$.
For some non-zero constant $c$, the map $$ \hat\rho:(\mathbb{P}^*,c~dx\wedge dy)\to (\chi(S^2,4)^*,\widehat \omega_{(S^2,4)})$$  
 is a $PSL(2,\ZZ)$ equivariant symplectomorphism.  
 \end{theoremABC}

The proof, as well as  as well as an exposition of the simpler case of $SU(2)$  character variety of the torus,  is contained in Section \ref{pillowcase}.

\medskip

The outline of the proof of the  flat case of  Theorem \ref{thm1} is the following.   Holonomy identifies flat moduli spaces with character varieties.  When $\cL$ is empty, Theorem \ref{thm1} follows from Weil's identification of the Zariski tangent spaces of character varieties with cohomology and Poincar\'e-Lefschetz duality.   Symplectic reduction  is used to  extend to the case when $\cL$ is nonempty.
We use only Poincar\'e duality with local coefficients, which is briefly reviewed, to highlight the fact  that the proof that the image of the differential is {\em maximal} isotropic (the subtlest part of any proof) does not depend on  the deeper result that the nondegenerate 2-form is closed, i.e., symplectic.   \medskip

To put Theorem \ref{thm1} in context, 
notice   that the Lagrangian immersion $r_U$   depends on the perturbation data $\pi$.  Building on ideas from \cite{MR974342,Flo,Weinstein} and many others, Wehrheim and Woodward developed quilts and Floer field theory \cite{FF_tangles,FF_coprime}; a framework that aims to produce a $2+1$ or, more generally, $(2,0)+1$ TQFT which factors as the composition of the (perturbed) flat moduli space functor, followed by passing to the Lagrangian Floer homology of the flat moduli spaces of surfaces equipped with the Atiyah-Bott form.  
  This can be considered as an approach to realizing a {\em bordered} Lagrangian Floer theory of character varieties of surfaces and 3-manifolds, as was done for Heegaard-Floer theory in \cite{LOT-main}.

 Even in the lucky case that one finds perturbation data $\pi$ for which  $\cM_\pi(X,\cL)$ is smooth of the correct dimension and  Lagrangian immerses into  the smooth stratum $\cM(\partial(X,\cL))^*$,   an understanding of how this immersion  depends on  $\pi$ is necessary  in order to extract topological information. 
If each perturbation curve is  parallel to an embedded curve in the boundary surface, for example, then varying the perturbation parameters changes $r$ by a Hamiltonian isotopy, and in particular has no effect on the topology of $\cM_\pi(X,\cL)$.   This is discussed in Section \ref{su2ahp}; see also \cite{HK}.

 However, for general choices of perturbation data,  perturbations typically change the topology of $\cM_\pi(X,\cL)$; indeed, the primary purpose of using perturbation curves is precisely to smooth $\cM_\pi(X,\cL)$.  As we discuss in Section~\ref{ssec:depend_pert},   Floer field theory posits that, nevertheless, the resulting immersions  should    be independent of $\pi$ in a Floer-theoretic sense (i.e., isomorphic in some Fukaya category).

\subsection{Acknowledgements} The authors thank  C. Judge, E. Toffoli,  and L. Jeffrey 
for helpful conversations. Special thanks to A. Kotelskiy, who read early versions of this article and made important suggestions.
\section{Review of Poincar\'e duality and symplectic forms}

\subsection{Preliminaries and Notation}  
\subsubsection{Symplectic linear algebra}  Let $A$ denote a finite-dimensional $\RR$ vector space.  A skew symmetric bilinear form $\omega:A\times A\to \RR$ is called a {\em symplectic form} if it is nondegenerate, that is, the map $A\to \Hom(A,\RR)$ given by $a\mapsto \omega(a,-)$ is an isomorphism. 
If $A$ admits a symplectic form, then its dimension is even;  denote the dimension by  $2n$. 

A {\em coisotropic subspace} $C\subset A$ is a subspace satisfying
$${\rm Annihilator}(C):=\{a\in A~|~ \omega(c,a)=0 \text{ for all } c\in C\}\subset C.$$
A {\em Lagrangian subspace} is a coisotropic subspace of dimension $n$, or, equivalently a subspace $L\subset A$ 
satisfying
 ${\rm Annihilator}(L)=L.$ 

{\em Symplectic reduction} refers the following process. Given any  coisotropic subspace $C\subset A$, the restriction
$$\omega|_C:C\times C\to \RR$$
may be degenerate, but $\omega|_C$ descends to a symplectic form $\widehat \omega$
on $\widehat{C}:=C/{\rm Annihilator}(C)$.  Furthermore if $L\subset A$ is any Lagrangian subspace of $(A,\omega)$ , then  
$$\hat L:=\big(L \cap C\big)/\big(L\cap{\rm Annihilator}(C) 
\big)\subset \widehat C$$ is  a Lagrangian subspace  of $(\widehat C,\widehat \omega)$.  The subquotient $(\widehat C,\widehat \omega)$ is called the {\em symplectic reduction of $(A,\omega)$ with respect to $C$}  and the Lagrangian subspace $\hat L\subset\widehat C$ the {\em symplectic reduction of $L$ with respect to $C$}.

 \subsubsection{The Lie group $G$}\label{LGG} Let $G$ be  connected compact Lie group.  Its  Lie algebra  $\frak{g}$ admits   a   positive definite symmetric $\ad$-invariant bilinear form, so we fix one and denote it by   $\{ \ , \ \}:\frak{g}\times \frak{g}\to \RR.$ Fix a conjugacy class $J\subset G$.    It is well known  that $G$ embeds in some $\RR^N$ as an algebraic variety.  Thus, if $F_g$ denotes  the free group on $g$ generators, $\Hom(F_g,G)=G^g$   is an affine real-set variety, with tangent space $\frak{g}^g$  at the trivial homomorphism. 
 
 \medskip

Our main focus, and the context for Theorem \ref{thm1}, concerns the case when $G=SU(2)$, $\{v,w\}=-\tfrac 1 2 {\rm Tr}(v w)$,  and 
$J\subset SU(2)$  is the conjugacy class of traceless matrices.  
To keep notation compact, we identify $SU(2)$ with the group of unit quaternions $\{a+b\bbi+c\bbj +d\bbk~|~a^2+b^2+c^2+d^2=1\}$ and
its Lie algebra $su(2)$ with the purely imaginary quaternions $\{ b\bbi + c\bbj+d\bbk\}$.   With this identification, $\Rea:SU(2)\to[-1,1]$ corresponds to $\frac 1 2 \Tr$.

\medskip

\subsubsection{Notation used for 2- and 3-manifolds} 
Throughout this article,   the notation 
$$Y,S,C,V, \text{ and } F$$ is fixed as follows.

First, $S$  denotes a possibly disconnected compact oriented surface without boundary.  Denote by $S_+$ the path connected based space obtained by first adding a disjoint base point to $S$, then attaching a 1-cell from this new base point to  each path component of $S$. 

\medskip

Next, $C$ denotes a finite disjoint union of $m$ circles in the surface $S$.  Number its components $C_i, i=1,\dots m$.  We assume that either each $C_i$ is oriented, or  
 the chosen conjugacy class $J\subset G$ is invariant under inversion in $G$, so that the condition that   a homomorphism from $\pi_1(S_+)\to G$ takes each loop $C_i$ into $J$ makes sense.  

\medskip

The pair $(S,C)$ determines a decomposition of $S$ into two surfaces as follows.

\medskip

Denote by $V$ a tubular neighborhood of $C$. Then $V\subset S$ is a disjoint union of  oriented annuli, one around each $C_i$.  
\medskip

Denote by $F$ the complementary surface, determining the decomposition:
\begin{equation}\label{deco1}S=F\cup_{ \partial F=\partial V}V.\end{equation}
 Orient $V$ and $F$ as subsurfaces of $S$. 
 
 \medskip
  
Finally, $Y$   denotes a compact, connected, and oriented 3-manifold with boundary $\partial Y=S$; that is, we assume that an orientation preserving identification of $\partial Y$ with $S$ is given.  Fix a base point in the interior of $Y$, and extend the inclusion $S=\partial Y\subset Y$ to a based embedding of $S_+$ into $Y$. 
 
\subsubsection{Tangles}\label{tanglenot}  A {tangle} $(X,\cL)$ consists of a connected compact oriented 3-manifold $X$, and $\cL\subset X$ a properly embedded compact 1-submanifold with boundary.   Thus $\cL$ consists of a disjoint union of $n$ intervals and $k$ interior circles.

A tangle $(X,\cL)$ gives rise to a triple $(Y,S,C)$ as above by taking 
$$Y=X\setminus N(\cL),$$
Where $N(\cL)$ denotes a tubular neighborhood of $\cL$, and letting $S=\partial Y$.  Then let $C\subset S$ denote a union of $m=n+k$ meridians of $\cL$, one for each component, viewed as a curve in $S$.   As before, $Y,S,C$ determine $V$ and $F$. Orientations of the tangle components are equivalent to orientations of the components of $C$.
 
 The process $(X,\cL)\Rightarrow (Y,S,C)$ is nearly reversible, by attaching 2-handles to $Y$ along $C$ and setting $\cL$ to be union of the co-cores of the 2-handles. The resulting tangle is obtained from $(X,\cL)$ by removing $k$ disjoint small balls from the interior of $X$,  each meeting a different closed component of $\cL$ in a trivial arc.
 This  results in a tangle with no closed components, but which has  the same character variety as the starting tangle.  We use the notation $(Y,S,C)$ in our arguments as it leads to simpler expressions, but state consequences in terms of the tangle $(X,\cL)$, as they are clearly seen as morphisms in a (more familiar) $(2,0)+1$ cobordism category.

\subsection{Poincar\'e duality and intersection forms}\label{pdif}

We begin by recalling the statement of Poincar\'e-Lefschetz duality with local coefficients for an oriented compact connected $n$-manifold $M$ with boundary $\partial M$ equipped with a finite cell decomposition (see  \cite{CLM} for a careful exposition and proofs, and \cite{MR141115} for an elementary derivation using dual regular cell decompositions). Denote by $\xi\in C_n(M,\partial M;\ZZ)$ a cellular chain representing the fundamental class.

Fix a base point in $M$ and some homomorphism $\pi_1M\to \Gamma$ to some group $\Gamma$, and denote by $M_\Gamma\to M$ the corresponding $\Gamma$ cover, equipped with the lifted cell structure and its cellular left $\Gamma$ covering action.   The cellular $\ZZ$ chain complexes of $M_\Gamma$ and $(M_\Gamma, \partial M_\Gamma)$ are denoted  by $C_*^\Gamma(M)$  and $C^\Gamma_*(M,\partial M)$. These are free finitely generated left $\ZZ\Gamma$ modules: a choice of $\ZZ\Gamma$ basis is given by arbitrarily choosing lifts of cells of $M$ (respectively of cells which meet $ M\setminus\partial M$).   A choice of cellular approximation of the diagonal determines a chain level cap product:
\begin{equation}\label{rawPD}
\cap \xi:\Hom_{\ZZ\Gamma}(C_*^\Gamma(M,\partial M),\ZZ\Gamma)\to C^\Gamma_*(M).
\end{equation}
The Poincar\'e duality theorem asserts that this map is a chain homotopy equivalence of  free $\ZZ\Gamma$  chain complexes.

 Given any left $\ZZ\Gamma$ module $W$, $\Hom_{\ZZ\Gamma}(C_*^\Gamma(M,\partial M),W)  =\Hom_{\ZZ\Gamma}(C_*^\Gamma(M,\partial M),\ZZ\Gamma)\otimes_{\ZZ\Gamma} W$ and hence  capping with $\xi$ induces a chain homotopy equivalence 
$$\cap \xi:\Hom_{\ZZ\Gamma}(C_*^\Gamma(M,\partial M),W)\to C^\Gamma_*(M)\otimes_{\ZZ\Gamma} W.$$
 
Now suppose that $W$ is a finite dimensional $\RR$ vector space equipped with a positive definite inner product $\{\ , \ \}:W\times W\to \RR$ and $\Gamma\to O(W)$ a representation, determining the left $\ZZ\Gamma$ module structure on $W$. There is an   algebraic chain isomorphism  of  $\RR$ chain complexes:
 $$
 C^\Gamma_*(M)\otimes_{\ZZ\Gamma} W\cong \Hom_{\RR}(\Hom_{\ZZ\Gamma}(C^\Gamma_*(M),W),\RR)
, ~ c\otimes w\mapsto\big( h\mapsto \{ h(c),w\}  \big). $$

Composing with  the chain homotopy equivalence $\cap \xi$, and using the universal coefficient theorem for $\RR$, we obtain an isomorphism
$$
H^*(M,\partial M ;W)\cong  H^*(\Hom_{\RR}(\Hom_{\ZZ\Gamma}(C^\Gamma_*(M),W),\RR))=
\Hom_{\RR}(H^*( M;W),\RR)
$$ 
 whose adjoint is the 
(by construction nondegenerate)  {\em cohomology intersection pairing over $W$} 
 \begin{equation}
\label{CNDP}H^*(M,\partial M;W )\times H^*(  M ; W )\to \RR.
\end{equation}
The cohomology intersection pairing can also be expressed in terms of cup  products:
 \begin{equation} 
 \label{cupprod} (x,y)\mapsto \{x\cup y\}\cap \xi,\end{equation}
 where $\{~\}:H^n(M;\partial M;W\otimes_{\ZZ\Gamma}W)\to H^n(M;\partial M;\RR)$
 is induced by the coefficient homomorphism $W\otimes_{\ZZ\Gamma}W\to\RR$  determined by the bilinear form $\{~,~\}$.

\medskip

When the boundary of $M$ is empty, this pairing, which we denote by 
 \begin{equation}
\label{CNDPemp}\omega_{M}: H^*(M;W )\times H^*(  M ; W )\to \RR
\end{equation}
 is therefore a  nondegenerate
inner product on the $\RR$ vector space $H^*(M;W )$.

If $\partial M$ is nonempty, precomposing the injective adjoint 
 $$H^*(  M ; W )\to \Hom_\RR(H^*(M,\partial M;W );\RR)$$ with the restriction map
 $H^*(M,\partial M;W)\to H^*(  M ; W )$
 yields a map 
 $$H^*(M,\partial M;W)\to \Hom_\RR(H^*(M,\partial M;W );\RR)$$
 with kernel equal to the kernel of $H^*(M,\partial M;W)\to H^*(  M ; W )$.  An equivalent statement is that 
the  pairing \begin{equation}\label{rel cup prod} \omega_{(M,\partial M)}:H^*(M,\partial M;W)\times H^*(M,\partial M;W)\to\RR\end{equation}
  has radical equal to $\ker H^*(M,\partial M;W)\to H^*(  M ; W )$.
 
 \medskip

Taking gradings into account,  when $\dim M=2n$  restriction defines a   pairing
  \begin{equation}
\label{CNDP2}\omega_{(M,\partial M)}:H^n(M,\partial M;W )\times H^n(  M ,\partial M; W )\to \RR
\end{equation}
with radical equal to $\ker H^n(M,\partial M;W)\to H^n(  M ; W )$.  

When $\dim M=4\ell +2$, for example when $M$ is a surface, the pairings
 $\omega_M$, and $ \omega_{(M,\partial M)} $ are skew-symmetric: 
$$\omega_M(x,y)=-\omega_M(y,x)~\text{ and }~ \omega_{(M,\partial M)}(x,y)=- \omega_{(M,\partial M)}(y,x).$$

\section{Two and three dimensional manifolds and symplectic linear algebra}
\subsection{ A symplectic form on the first cohomology of surface}

Recall that  $S_+$ denotes the path connected CW complex obtained by adding a disjoint base point  to the oriented surface $S$ and  a 1-cell  connecting   each
path component of $S$ to this base point.  Let $\Gamma=\pi_1(S_+)$.  Its universal cover
$\widetilde S_+\to S_+$ is a regular $\Gamma$ cover, and hence so is its restriction over $S$
$$\widetilde S\to S.$$
A representation
 $\rho:\Gamma\to G$  is fixed.

\medskip

Then $\rho$ determines, via the adjoint action of $G$, the representation $\ad\rho:\Gamma\to O(\frak{g})$, and hence cohomology groups $H^*(S;\frak{g}), H^*(S,\partial S;\frak{g}), $ and $H^*(\partial S;\frak{g})$, with, for example,
$$
H^*(S;\frak{g}):=H^*(\Hom_{\ZZ\Gamma}(C^\Gamma_*(S), \frak{g})).
$$
 If we wish to emphasize $\rho$, we write $H^*(S;\frak{g}_{\ad \rho})$.

 Equation (\ref{CNDPemp}) shows that since the boundary of $S$ is empty, 
 \begin{equation}
\label{closedcase}
\omega_S:H^1(S;\frak{g})\times H^1(S;\frak{g})\to \RR
\end{equation}
  is a nondegenerate skew-symmetric form. 

\medskip

On the other hand,  if $C$ is nonempty, so that the boundary of $F$ is nonempty, Equation (\ref{CNDP2}) shows that 
$$\omega_{(F,\partial F)}:H^1(F,\partial F;\frak{g} )\times H^1(  F ,\partial F;\frak{g})\to \RR
$$
is  in general a degenerate skew symmetric form with radical equal to 
$$\ker H^1(F,\partial F;\frak{g})\to H^1(  F ;\frak{g} ).$$

A degenerate skew symmetric form induces a nondegenerate form on the quotient by its radical.  The exact sequence of the pair $(F,\partial F)$ shows that $\omega_{(F,\partial F)}$ descends to a {\em nondegenerate} skew symmetric form
 \begin{equation}
\label{boundarycase}\widehat\omega_F:\widehat{H}^1(  F ;\frak{g} )\times
\widehat{H}^1(  F ;\frak{g} )
\to \RR,
\end{equation}
  where 
$$\widehat{H}^1(  F ;\frak{g} )={\rm Image } \ H^1(F,\partial F;\frak{g})\to H^1(  F ;\frak{g} )=
\ker H^1(  F ;\frak{g} )\to H^1(\partial F;\frak{g}).$$
Summarizing:
\begin{proposition} Let $S$ be a compact oriented surface without boundary and $\rho:\pi_1S_+\to G$  a representation.  Then  $(H^1(S;\frak{g}_{\ad\rho}),\omega_S)$ is a symplectic vector space.   If $F\subset S$
is the complement of a nonempty disjoint union of annuli $V$,    then  $(\widehat{H}^1( F ;\frak{g}_{\ad\rho} ),\widehat \omega_{F})$ is a symplectic vector space.
\end{proposition}

In the following  diagram of inclusions, the two vertical and two horizontal rows are exact, and the isomorphisms excisions: 
\begin{equation}\label{badass}
\begin{tikzcd}
             &                                  &H^1(S,F) \arrow[d, "q_1"] \arrow[r, "\cong"] & H^1(V,\partial V) \arrow[d, "q_2"]           &             \\
{} \arrow[r] & H^1(S,V) \arrow[r, "\alpha"] \arrow[d,"p_0", "\cong"'] & H^1(S) \arrow[r, "\beta"] \arrow[d, "p_1"] & H^1(V)\arrow[r,"\gamma"] \arrow[d, "p_2"] & H^2(S,V) \arrow[d,"\cong"] \\
{} \arrow[r] & H^1(F,\partial F) \arrow[r, "a"']                & H^1(F) \arrow[r, "b"']                & H^1(\partial F) \arrow[r,"c"']                & H^2(F,\partial F)          
\end{tikzcd}
\end{equation}
\begin{proposition}\label{prop1.2} The kernel of $\beta$ is a cosiotropic subspace of $H^1(S;\frak{g})$ with annihilator $\ker p_1$, and hence $p_1$ induces a symplectomorphism
$$\ker \beta/\ker p_1\xrightarrow{\cong} \widehat{H}^1(F;\frak{g}).$$
 
\end{proposition}
 \begin{proof}  The surface $V$ is a disjoint union of annuli.  The composition  $\{0\}\times  S^1\subset \partial(I\times S^1)\subset I\times S^1$  a homotopy equivalence, and hence the restriction $H^1(I\times S^1)\to H^1(\partial(I\times S^1))$ is injective with any coefficients.  Hence $p_2$ is injective, and  $q_2=0$.

  If $s\in H^1(S)$ satisfies $\omega_S(s,\alpha(y))=0$ for all $y\in H^1(S,V)$, then 
 $\omega_{F,\partial F}(p_1(s), a\circ p_0(y))=0$ for all $y\in H^1(S,V)$. Hence
 $\omega_{F,\partial F}(p_1(s), a(z))=0$ for all $z\in H^1(F,\partial F)$. Since the pairing
 $H^1(F,\partial F)\times H^1(F)\to \RR$ is nondegenerate (Equation (\ref{CNDP})), this implies that $p_1(s)=0$.
 In other words, the annihilator of $\text{image }\alpha=\ker \beta$ is contained in $\ker p_1$.
 
  Since $q_2=0$, $\ker p_1\subset\ker\beta$. 
Therefore   $\ker \beta$ contains its annihilator and hence is coisotropic.

It remains to show that $\ker p_1=\text{image }q_1$ is contained in the annihilator of $\ker \beta$. 
 Given $x\in H^1(S,V)$ and $y\in H^1(S,F)$, 
 $\omega_S(\alpha(x), q_1(y))=0$ 
 since the cup product
 $$H^1(S,V)\times H^1(S,F)\to H^2(S)$$ factors through $H^2(S,F\cup V)=0$
(see Equation (\ref{cupprod})).
 \end{proof}
 
\begin{corollary}\label{reduction}
If   $L\subset H^1(S;\frak{g}_{\ad \rho})$ is any Lagrangian subspace, then the  image
$$p_1(L\cap \ker \beta)\subset   \widehat{H}^1(F;\frak{g}_{\ad \rho})$$
is a Lagrangian subspace.
\end{corollary}

\subsection{Restriction from a 3-manifold with boundary}
Recall that $Y$ is a compact, connected, oriented
3-manifold with boundary $S=\partial Y$, extended to  an embedding  $S_+\subset Y$.   

Assume that the
representation $\rho:\pi_1(S_+)\to G$ is a restriction of a representation  (of the same name) $\rho:\pi_1(Y)\to G$.

\begin{lemma}\label{lemma1} The image of the restriction map, 
$$L_Y:={\rm Image } \ r:H^1(Y;\frak{g}_{\ad \rho})\to H^1(S;\frak{g}_{\ad \rho})$$
is a Lagrangian subspace of $(H^1(S;\frak{g}_{\ad \rho}),\omega_S)$.
\end{lemma}
\begin{proof}
In the following diagram, the middle row is part of the exact sequence of the pair $(Y,S)$. 
 The vertical arrows all isomorphisms, with the downward pointing isomorphisms Poincar\'e-Lefschetz duality. The diagram commutes up to sign \cite{MR1325242}. We suppress the $\frak{g}_{\ad\rho}$ coefficients.

\[
\begin{tikzcd}
 \Hom(H_1(Y);\RR) \arrow[r]                     & \Hom(H_1(S);\RR)  \arrow[r,"\delta^*"]     &\Hom(H_2(Y,S);\RR)  \\
 H^1(Y) \arrow[r,"r"] \arrow[d] \arrow[u] & H^1(S) \arrow[d] \arrow[u] \arrow[r,"\nu"]&H^2(Y,S) \arrow[d] \arrow[u] \\
 H_2(Y,S) \arrow[r,"\delta"]                     & H_1(S) \arrow[r]&H_1(Y)                   
\end{tikzcd}
\]
The  image of $r$ equals the kernel of $\nu$, which is isomorphic to the kernel of $\delta^*$. The kernel of $\delta^*$ is isomorphic to the cokernel of its dual $\delta$, which in turn is isomorphic to the cokernel of $r$. Hence the image and cokernel of $r$ are isomorphic, and so  $ \dim(\text{image}(r))=\tfrac{1}{2}\dim H^1(S)$.

Commutativity of the following diagram is a consequence of naturality of cup product and Poincar\'e duality:

\[
\begin{tikzcd}
H^1(Y;\frak{g})\times H^1(Y;\frak{g}) \arrow[r,"\cup_{\{  ,  \}}"] \arrow[d, "r\times r"] & H^2(Y;\RR) \arrow[r,    "\cap{[Y,S]}"           ] \arrow[d] & H_1(Y, S;\RR) \arrow[d] \arrow[rd, dotted] &   \\
H^1(S;\frak{g})\times H^1(S;\frak{g}) \arrow[r,"\cup_{\{  ,  \}}"]           & H^2(S;\RR) \arrow[r,    "\cap{[S]}"           ]         & H_0( S;\RR)  \arrow[r,"\ep"] \arrow[d]  & \RR \\
                      &                       & H_0( Y;\RR) \arrow[ru,"\ep"']           &  
\end{tikzcd}
\]
Exactness  of the vertical sequence   shows that the dotted arrow is
zero, which implies that the image of $r$ is isotropic, and therefore Lagrangian. 
 \end{proof}

Recall that the boundary $S=\partial Y$ is equipped with a
embedded collection $C\subset S$ of circles, with tubular neighborhood $V$ and complementary subsurface $F$, producing
the 
decomposition $S=F\cup V$ as in Equation (\ref{deco1}).  
Consider the ladder of exact sequences, with  all maps induced by inclusions.  The $\frak{g}_{\ad\rho}$ coefficients are supressed. The bottom two rows coincide with those of Diagram (\ref{badass}).
\begin{equation}\label{diag2}
\begin{tikzcd}
\cdots \arrow[r]& H^1(Y,V) \arrow[r,"A" ] \arrow[d, "r_0"] & H^1(Y) \arrow[r,"B"  ] \arrow[d,"r"] & H^1(V) \arrow[d,equal]  \arrow[r]&\cdots\\
\cdots\arrow[r]&H^1(S,V)\arrow[r, "\alpha"]\arrow[d,"p_0", "\cong" '] & H^1(S)\arrow[r,"\beta"]\arrow[d,"p_1"]&H^1(V)\arrow[d,"p_2"]\arrow[r]&\cdots\\
\cdots\arrow[r]&H^1(F,\partial F) \arrow[r,"a" ] & H^1(F)\arrow[r,"b"]&H^1(\partial F)\arrow[r]&\cdots
\end{tikzcd}
\end{equation}
A diagram chase shows   that ${\rm Image }\ r\cap \ker \beta={\rm Image }\ \alpha\circ r_0$.
Hence 
$$
p_1({\rm Image }\ r\cap \ker \beta)={\rm Image } \ a\circ p_0\circ r_0:H^1(Y,V)\to H^1(F) 
\subset \ker b=\widehat H^1(F)$$

Lemma \ref{lemma1} and  Corollary \ref{reduction} imply the following.

\begin{corollary} \label{thepoint} Suppose that $\partial Y=S=V\cup F$ with  $V$  a disjoint union of annuli.

Then 
$$ 
L_{Y,V}:= {\rm Image } \  H^1(Y,V;\frak{g}_{\ad \rho})\to H^1(F;\frak{g}_{\ad \rho})\subset  \widehat{H}^1(F;\frak{g}_{\ad \rho})
$$ is  a Lagrangian subspace. Moreover,  $L_{Y,V}$ is the symplectic reduction of $L_Y$ with respect to $\ker \beta:H^1(S;\frak{g}_{\ad \rho})\to H^1(V;\frak{g}_{\ad \rho})$.
 
\end{corollary}

\section{Character varieties, relative character varieties, and their tangent spaces}

 \subsection{Character varieties}

 \begin{definition}
 Given a finitely presented group $\Gamma$, its {\em $G$ character variety} $\chi_G(\Gamma)$  is the real semi-algebraic set   defined to be the orbit space of the $G$-conjugation action on the affine $\RR$-algebraic set $\Hom(\Gamma,G)$. A set of $g$ generators of $\Gamma$ embeds $\Hom(\Gamma,G)$ in $G^g$ equivariantly, and, since $G$ is a compact Lie group, $\Hom(\Gamma,G)$ is an affine $\RR$-algebraic set, with orbit space $\chi_G(\Gamma)$ a semi-algebraic set.  We call $\chi_G(\Gamma)$ and (orbit spaces of conjugation invariant) Zariski closed subsets {\em character varieties.}  If  $Z$ is a path connected space, write $\chi_G(Z)$ instead of $\chi_G(\pi_1(Z))$.
The character variety of a non-path connected space, {\em by definition}, is the cartesian product of the character varieties of its path components.     
   \end{definition}

 \medskip

As observed by Weil  \cite{MR169956}, the {\em  formal  tangent space at }
 the conjugacy class of  {\em any} representation $\rho:\Gamma\to G$ to the character variety  $\chi_G(\Gamma)$   is naturally identified with  first cohomology: \begin{equation}\label{FTS}
T_\rho\chi_G(\Gamma)=H^1(\Gamma;\frak{g}_{\ad \rho}).
\end{equation}
We take Equation (\ref{FTS}) as the definition of the formal tangent space at $\rho$ for any $\rho\in \chi_G(\Gamma)$. Recall that for any space $X$ with fundamental group $\Gamma$, $H^1(\Gamma;\frak{g}_{\ad \rho})$ and $H^1(X;\frak{g}_{\ad \rho})$ are canonically isomorphic.

Weil's argument is based on the calculation that if  a path of representations is expressed in the form $\rho_s=\exp(\alpha_s)\rho_0$ for some path 
$\alpha_s:\Gamma\to \frak{g}$, then $\tfrac{d}{ds}|_{s=0}~\alpha_s:\Gamma\to\frak{g}$ is a 1-cocycle \cite{MR1324339}.

\subsection{Relative character varieties}
As above, assume that $Y$ is a compact connected 3-manifold
with boundary $S=\partial Y$,  $C\subset S$ is a  union of $m$ circles $C_i$, $V$ is the tubular neighborhood of $C$ in $S$, with complementary surface $F$.  

Assume further that either $J$ is invariant via the inversion map of $G$ (as is the case for the conjugacy class of traceless matrices in $SU(2)$) or else assume that every circle   $C_i$ is equipped with an orientation.

 Then define the {\em relative character variety} to be
 \begin{equation}
\label{relchi}
\chi_{G,J}(Y,C)\subset \chi_G(Y)
\end{equation}
  to be  the subvariety consisting of conjugacy classes of representations $\pi_1(Y)\to G$ which send   (any based representative of the  homotopy class of) each circle in $C$ into $J$. 
 Define the {\em formal tangent space}  of $\chi_{G,J}(Y,C)$   at $\rho$ to be
 \begin{equation}\label{RFTS}
T_\rho\chi_{G,J}(Y,C)= \ker H^1(Y;\frak{g}_{\ad\rho})\to H^1(C;\frak{g}_{\ad\rho})
\end{equation}

Given an oriented surface $F$,   define  
$$ \chi_{G,J}(F,\partial F)=\chi_{G,J}(F\times [0,1], \partial F\times \{\tfrac 1 2\}).$$ 
   Its formal tangent space is
\begin{equation}\label{RSFTS}
T_\rho\chi_{G,J}(F,\partial F)= \ker H^1(F;\frak{g}_{\ad\rho})\to H^1(\partial F;\frak{g}_{\ad\rho})=\widehat H^1(F;\frak{g}_{\ad\rho})\end{equation}

\medskip

 \subsection{Regular points}\label{regpts}
The term ``formal tangent space'' may be replaced by its usual elementary differential topology notion in neighborhoods of  {\em regular points}  of   $\chi_{G,J}(Y,C)$ and $\chi_{G,J}(F,\partial F)$, as we next explain. In brief, as elsewhere in gauge theory, a regular point is one which has a neighborhood diffeomorphic to Euclidean space 
of the correct index-theoretic dimension.  We provide a stripped-down explanation, suitable for our purposes,  of what this means, for the benefit of the reader.

First, given a connected compact surface $F$ (with possibly empty boundary) we call $\rho\in \chi_{G,J}(F,\partial F)$ a regular point provided $\rho$ has a neighborhood $U\subset \chi_{G,J}(F,\partial F)$ so that 
 $\dim \widehat H^1(F;\frak{g}_{\ad\rho'})$ 
is independent of $\rho'\in U$. 

Next, given a disjoint union of connected compact surfaces $F=\sqcup_iF_i$, call
 $$\rho\in \chi_{G,J}(F,\partial F)=\prod_i\chi_{G,J}(F_i,\partial F_i)$$  a regular point
 provided each of its components is a regular point. 
 
 Finally,  for a pair $(Y,C)$ (with $C$ possibly empty) $\rho\in  \chi_{G,J}(Y,C)$ is called a regular point provided $\rho$ admits
 a neighborhood  $U\subset \chi_{G,J}(Y,C)$ so that for all $\rho'\in U$:
 \begin{itemize}
\item The restriction map $\chi_{G,J}(Y,C)\to \chi_{G,J}(F,\partial F)$ takes $\rho'$ to a regular point, and
\item $$\dim T_{\rho'}\chi_{G,J}(Y,C)=\tfrac 1 2 \dim T_\rho\chi_{G,J}(F,\partial F)
=\tfrac 1 2 \sum_{i}  \dim T_\rho\chi_{G,J}(F_i,\partial F_i)$$
\end{itemize}

Hence, if $\rho\in  \chi_{G,J}(Y,C)$ is a regular point, the map $\chi_{G,J}(Y,C)\to \chi_{G,J}(F,\partial F)$, which takes a representation of a 3-manifold group to its restriction to the boundary surface, is, near $\rho$, a smooth map of a smooth $n$-disk into $\RR^{2n}$ for some $n$.
\medskip

    Notice that $\chi_{G,J}(Y,C)$ is the preimage of the point $(J,\cdots, J)$
   under the restriction map
   $$\chi_G(Y)\to \chi_G(C)=\prod_{i=1}^m\chi_G(C_i)= (G/_{\rm conjugation})^m.$$
and that $\chi_{G,J}(F,\partial F)$ is the preimage of the point $(J,\cdots, J)$
   under the restriction map
   $$\chi_G(F)\to \chi_G(\partial F).$$

Since $C\subset V$ is a deformation retract, the exact sequence of the pair shows
that
\[
T_\rho \chi_{G,J}(Y,C)\cong {\rm Image} \ H^1(Y,V;\frak{g}_{\ad\rho})\to H^1(Y;\frak{g}_{\ad\rho}).
\]

The image of the differential of the restriction map $\chi_{G,J}(Y,C)\to \chi_{G,J}(F,\partial F)$ at $\rho\in 
\chi_{G,J}(Y,C)$ is therefore identified with $L_{Y,V}$, the image of the composition
$$
H^1(Y,V;\frak{g}_{\ad\rho})\to H^1(F,\partial F;\frak{g}_{\ad\rho})\to \widehat H^1(F;\frak{g}_{\ad\rho}),
$$
which by Corollary \ref{thepoint} is a Lagrangian subspace of $(H^1(F;\frak{g}_{\ad\rho}), \widehat \omega_F)$.  Summarizing:  
\begin{corollary}
  \label{tm4} If $\rho\in \chi_{G,J}(Y,C)$ is a regular point,  then there exists a neighborhood of $\rho$ so that the differential of the restriction $\chi_{G,J}(Y,C) \to  \chi_{G,J}(F,\partial F)$ at any point $\rho'$  in this neighborhood has image a Lagrangian subspace of $(\widehat H^1(F;\frak{g}_{\ad\rho'}),\widehat\omega_F)$.\end{corollary}

  \medskip

It is known that for surfaces, with the exception of a few low genus cases, the regular points coincide with the irreducible representations.

\color{black}

\subsection{Symplectic structure}

The proof  of Corollary \ref{tm4}  does not rely of the following fundamental result  of Atiyah-Bott \cite{Atiyah_Bott} and its extensions due to   Goldman  \cite{Goldman},  Karshon \cite{MR1112494}, Biswas-Guruprasad \cite{MR1232983}, King-Sengupta \cite{MR1372813},  Guruprasad-Huebschmann-Jeffrey-Weinstein \cite{MR1460627}.

\begin{theorem}\label{AB}
On the top stratum  of   regular points of $\chi_G(S)$ and $\chi_{G,J}(F,\partial F)$, the 2-forms $\omega_S$ and $\widehat \omega_F$ are closed, that is,  are symplectic forms.
\end{theorem}
  
  In light of this result, Corollary \ref{tm4} can be restated as follows.
 
\begin{theorem}
 \label{flatthm}
Suppose that $\rho\in\chi_{G,J}(Y,C)$ is a regular
  point.
  Then there exists a neighborhood $U$ of $\rho$ in
  $\chi_{G,J}(Y,C)$ so that the restriction of $r$ to $U$, 
  $$r|_U :U\to \chi_{G,J}(F,\partial F),$$  is a  Lagrangian embedding.
  
    In particular, if $\chi_{G,J}(Y,C)$ contains only  regular points, then the restriction map is a Lagrangian immersion.
\end{theorem}

Given a tangle $(X,\cL)$, We write $\chi_{G,J}(X,\cL)$ rather than 
 $\chi_{G,J}(Y,C)$ where $Y,S,C,F$ and $V$ are determined by $(X,\cL)$ as in Section \ref{tanglenot}. Also write $\chi_{G,J}(\partial(X,\cL))$ rather than $\chi_{G,J}(F, \partial F)$.   Then Theorem \ref{flatthm} can be restated in the new notation as follows.

\begin{corollary} \label{flatthmtangle}
Suppose that $\rho\in\chi_{G,J}(X,\cL)$ is a regular
  point.  Then there exists a neighborhood $U$ of $\rho$ in
  $\chi_{G,J}(X,\cL)$ so that the restriction of $r$ to $U$
  $$r|_U :U\to \chi_{G,J}(\partial(X,\cL)),$$ 
   is a  Lagrangian embedding.
  
  In particular, if $\chi_{G,J}(X,\cL)$ contains only regular points, then the restriction map is a Lagrangian immersion.
\end{corollary}

  In what follows,   we simply write $\chi(A)$ for $\chi_G(A)$ and $\chi(A,B)$ for $\chi_{G,J}(A,B)$.
In addition, we denote by $\chi(S)^*\subset \chi(S)$ and $\chi(F,\partial F)^*\subset \chi(F,\partial F)$ the {\em smooth top strata}, as real algebraic varieties, equipped with the symplectic forms $\omega_S$, $\widehat \omega_F$. 

\section{Three pillowcases}\label{pillowcase}
In this section, take $G=SU(2)$, and let $J\subset SU(2)$ be the conjugacy class of unit quaternions
with zero real part, so $J=su(2)\cap SU(2)$, a 2-sphere.
\subsection{The quotient of the torus by the elliptic involution}
Let $\mathbb{T}$ denote $U(1)\times U(1)$ with its symplectic form $dx \wedge dy$ and symplectic $SL(2,\ZZ)$ action given in Equation \ref{SL2Z}. 
Multiplication by $-1\in SL(2,\ZZ)$ is central and hence the quotient 
$$\mathbb{P}=\mathbb{T}/\{\pm 1\}$$
inherits a $PSL(2,\ZZ)=SL(2,\ZZ)/\{\pm1\}$ action.  The quotient map $\mathbb{T}\to \mathbb{P}$ is the 2-fold branched cover of the 2-sphere with branch points the four points $\{(\pm 1, \pm 1)\}$.  The complement, $\mathbb{P}^*$  of the four branch points  is a smooth surface, with symplectic form $dx\wedge dy$ and symplectic 
$PSL(2,\ZZ)$ action, and the restriction 
$$\mathbb{T}^*\to\mathbb{P^*}$$
is a smooth symplectic 2-fold covering map.

 \subsection{The   $SU(2)$ character variety of the genus 1 surface} 
 

Consider the genus one closed oriented surface $T$, equipped with generators $\mu,\lambda\in\pi_1T$ represented by a pair of oriented loops intersecting geometrically and algebraically once. The Dehn twists $D_\mu,D_\lambda$ about $\mu$ and $\lambda$
induce the automorphisms 
$$
D_\mu: \big( \mu\mapsto\mu, \lambda\mapsto \mu \lambda\big) \text{ and }
D_\lambda:\big( \mu\mapsto\mu \lambda^{-1}, \lambda\mapsto  \lambda\big) .
$$
of $\pi_1(T,t_0)\cong \ZZ  \mu\oplus\ZZ\lambda$ (where $t_0$ lies outside the support of these 2 Dehn twists).  These automorphisms induce, by precomposition,  homeomorphisms
$$D_\mu^*,D_\lambda^*:\chi(T)\to \chi(T).$$

Let $\mathcal{H} (T)=\Hom(\pi_1(T), SU(2))$.
Denote by $$p:\mathcal{H} (T)\to \chi (T)$$ the (surjective) orbit map of the conjugation action.

 Define 
\begin{equation}\label{tro}\rho:\mathbb{T}\to \mathcal{H}(T), ~ \rho(e^{x\bbi}, e^{y\bbi})= \big ( \mu\mapsto e^{x\bbi}, ~ \lambda\mapsto e^{y\bbi}\big).\end{equation}

\begin{proposition}\label{pillow1}  The Dehn twists $D_\mu$ and $D_\lambda$ define a symplectic $PSL(2,\ZZ)$ action on $(\chi(T),\omega_T)$.
 The
 map $\rho$ of Equation (\ref{tro}) descends to a $PSL(2,\ZZ)$ equivariant homeomorphism
 $$\rho:\mathbb{P}\to \chi(T)$$
 which restricts to  a $PSL(2,\ZZ)$ equivariant  symplectomorphism
 $$ (\mathbb{P}^*,c~ dx\wedge dy)\to (\chi(T)^*,\omega_T) $$
 on the top strata of the $SU(2)$ character varieties.
\end{proposition}
\begin{proof}
It is well known that $\rho:\mathbb{P}\to\chi(T)$ is a well defined homeomorphism, as well as an analytic diffeomorphism of the smooth strata $\mathbb{P}^*\to \chi(T)^*$. This follows simply from the observations that the fundamental group of $T$ is abelian   and the Weyl group of $SU(2)$  is $\ZZ/2$.

We show that   $\rho:(\mathbb{P}^*,dx\wedge dy)\to (\chi(T)^*,\omega_T)$ is a symplectomorphism. Since this is a local statement, we work in $\RR^2$ for simplicity.
Define $m=\rho\circ e:\RR^2\to \chi(T)$, where $e(x,y)=(e^{x\bbi}, e^{y\bbi})$.

The differential  of $m$ at $(x,y)\in \RR^2$, $dm:T_{(x,y)}\RR^2\to T_{m(x,y)}(SU(2)\times SU(2))$ is given by
$$dm(\tfrac{\partial}{\partial x})=\big( \mu\mapsto e^{x\bbi}\bbi, \lambda\mapsto 0\big), dm(\tfrac{\partial}{\partial y})=\big( \mu\mapsto 0, \lambda\mapsto e^{y\bbi}\bbi\big).
$$
Following Weil, left translation in $SU(2)\times SU(2)$ identifies these with the $ su(2)$-valued 1-cochains
 $$z_x=(\mu\mapsto \bbi, \lambda\mapsto 0), ~
z_y=(\mu\mapsto 0, \lambda\mapsto \bbi)\in C^1(T;su(2)_{\ad m(x,y)}).$$

The subspaces $L=\bbi\RR$ and $V=\bbj \RR +\bbk \RR$ are invariant and complementary with respect to $\ad m(x,y)$, and therefore
$$H^1(T;su(2)_{\ad m(x,y)})=H^1(T; L_{\ad m(x,y)})\oplus H^1(T; V_{\ad m(x,y)}).$$
Note that the action $\ad m(x,y)$ on $L$ is trivial, since $L=\bbi \RR$ and $m(x,y)$ has values in the abelian subgroup $\{e^{\bbi u}\}$. Hence 
$$H^1(T; L_{\ad m(x,y)})\cong H^1(T; \RR)\cong\RR^2.$$
The branched cover $ m:\RR^2\to \chi(T)$  is a local diffeomorphism near any $(x,y)\in (\RR^2)^*=\RR^2\setminus (\pi\ZZ)^2$, and hence it follows that for such $(x, y)$, 
$$H^1(T;\RR)\otimes \RR\bbi=H^1(T; L_{\ad m(x,y)})=\Span\{ z_x,z_y\}\text{ and } H^1(T; V_{\ad m(x,y)})=0$$
(these calculations can  also be easily checked directly), so that 
$$T_{m(x,y)}(\chi (T))=H^1(T;\RR)\otimes \RR\bbi.$$

The cup product
$$H^1(T; L_{\ad m(x,y)})\times H^1(T; L_{\ad m(x,y)})\to H^2(T; \RR)$$
is thereby identified with  
$$H^1(T;\RR)\otimes\RR\bbi \times H^1(T;\RR)\otimes\RR\bbi\to H^2(T; \RR)$$
$$ (z_1\otimes  \bbi,z_2\otimes  \bbi)=(z_1\cup z_2)\cdot (-\Rea( \bbi  \bbi))=z_1\cup z_2.$$
In particular, $z_x=\mu^*\otimes \bbi $ and $z_y=\lambda^*\otimes \bbi$, where $\mu^*,\lambda^*\in H^1(T;\ZZ)=\Hom(H_1(T;\ZZ),\ZZ)$ is the basis dual to $\mu,\lambda$. This basis is symplectic with respect to the (usual, untwisted) intersection form.  Hence
$$
\big((p\circ m)^*(\omega_T)\big)|_{(x,y)}(\tfrac{\partial}{\partial x} , \tfrac{\partial}{\partial y}) =\omega_T(z_x,z_y)= 
(\mu^*\cup\lambda^*) \cap [T]=1=   dx\wedge dy (\tfrac{\partial}{\partial x} , \tfrac{\partial}{\partial y})=1.$$

\medskip

Naturality of cup products shows that  $D_\mu^*$, $D_\lambda^*$ preserve the symplectic form $\omega_T$. 
Next,  
 $$D^*_\mu(m(x,y))=m(x,y)\circ D_\mu=\big( \mu\mapsto m(x,y)(\mu)=e^{x\bbi}, \lambda\mapsto m(x,y)(\mu\lambda)= e^{(x+y)\bbi}\big)
$$ 
and similarly
 $$D^*_\lambda(m(x,y))=m(x,y)\circ D_\lambda=\big( \mu\mapsto  e^{(x-y)\bbi}, \lambda\mapsto  e^{y\bbi}\big)
$$ 

It follows that the subgroup of Homeo$(\chi(T))$ generated by $D^*_\mu$ and $D^*_\lambda$ pulls back, via the homeomorphism $\rho:\mathbb{P}\to \chi(T)$, to the subgroup of $PSL(2,\ZZ)$ generated by 
\begin{equation*} 
\rho^*(D^*_\mu)=\begin{pmatrix} 1&1\\ 0&1 \end{pmatrix},~
\rho^*(D^*_\lambda)=\begin{pmatrix} 1&0\\-1&1 \end{pmatrix}.
\end{equation*}
These two matrices generate $PSL(2,\ZZ)$ (\cite{MR607504}), finishing the proof.
 \end{proof}

\color{black}

\subsubsection{The solid torus and the restriction to its boundary}\label{solid torus}  
Let $X$ denote the solid torus with boundary $T$. Equip $T$ with based loops $\mu,\lambda$ generating $\pi_1(T)$, so that $\mu$ is trivial in $\pi_1(X)$  and $\lambda$ generates $\pi_1(X)$. Then
 $$\chi(X) =\chi(\ZZ\lambda)=SU(2)/_{\rm conjugation}.$$ 

    An explicit slice of the conjugation action $\Hom(\ZZ\lambda,SU(2))\to \chi(\ZZ\lambda)$
 is given by the map 
 \begin{equation}
\label{arc}
[0,\pi]\to \Hom(\ZZ\lambda ,SU(2)),~ s\mapsto  (\lambda\mapsto e^{\bbi s})\end{equation}
with composition $[0,\pi]\to \chi(X)\xrightarrow{\Rea_\lambda}[-1,1]$ equal to the analytic isomorphism $\cos(s)$.  
Simple cohomology calculations show
 $\dim H^1(\ZZ;su(2))$  equals 1 when $0<s<\pi$ and equals $3$ when $s=0$ or $\pi$. This shows that 
the interior of the interval forms the smooth top stratum  of $\chi(\ZZ)$, and the endpoints are singular.
 
 \medskip

  Since $\mu=1\in \pi_1(X)$, the restriction-to-the boundary map
 \begin{equation}
\label{eq4.5}\chi(X)\to \chi(T)
\end{equation}
 is  is easily computed, in $\mathbb{P}$, to be the smooth (necessarily Lagrangian) embedded arc given by:
  \begin{equation}
\label{eq4.6}[0,\pi]\ni s \mapsto [e^{s\bbi},1]\in  \mathbb{P}
\end{equation}
with endpoints at $[-1,1]$ and $[1,1]$.

\subsection{The traceless $SU(2)$ character variety of the 4-punctured 2-sphere}

Let $4D^2\subset S^2$ be four disjoint open disks, and set
$$F=S^2\setminus 4 D^2, ~ \partial F= 4 S^1.$$
Then $\pi_1(F)$ has presentation
$$\pi_1(F)=\langle a,b,c,d ~|~ abcd=1\rangle$$
and is free on $a,b,c$.
Set 
$$\mathcal{H}(S^2,4)=\{\rho\in\Hom(\pi_1(F),SU(2))~|~ \rho(a),\rho(b),\rho(c),\rho(d)\in J\}.$$   Let
$p:\mathcal{H} (S^2,4)\to \chi (S^2,4)$ denote the orbit map of the conjugation action.

 Define 
\begin{equation}\label{trho2} {\hat \rho}:\mathbb{T}\to \mathcal{H} (S^2,4), ~ (e^{x\bbi}, e^{y\bbi})\mapsto \big ( a\mapsto \bbj, ~ b\mapsto e^{x\bbi}\bbj,~ c\mapsto e^{y\bbi}\bbj\big).\end{equation}

A half Dehn twist of $(D^2, p,q)$, where $p,q$ are a fixed pair of interior points, is a homeomorphism of the disk which fixes the boundary and permutes $p$ and $q$, and which generates the infinite cyclic mapping class group rel boundary of a disk with two marked interior points.
The generator which veers to the right is called a positive half Dehn twist.

\subsubsection{Proof of Theorem \ref{PC=PC2}}
We prove the Theorem \ref{PC=PC2} of  the introduction.

\medskip

\begin{proof}

That $\hat\rho:\mathbb{P}\to\chi(S^2,4)$ is a well defined homeomorphism, as well as an analytic diffeomorphism of the smooth strata $\mathbb{P}^*\to \chi(S^2,4)^*$,  is simple;
its proof can be found in \cite{Lin, HeuKro}.

We show that $\hat\rho:(\mathbb{P}^*,c~ dx\wedge dy)\to (\chi(S^2,4)^*,\widehat \omega_{(S^2,4)})$ is a symplectomorphism, for some constant $c$.  Since this is a local statement, we work in $\RR^2$ for simplicity.
Define $\hat m=\hat\rho\circ e:\RR^2\to \chi(S^2,4)$, where $e(x,y)=(e^{x\bbi}, e^{y\bbi})$.

 The differential  of $\hat m$ at $(x,y)\in \RR^2$ is given by
$$d\hat m(\tfrac{\partial}{\partial x})=\big( a\mapsto 0, b\mapsto e^{x\bbi}\bbi\bbj,c\mapsto 0\big), d\hat m(\tfrac{\partial}{\partial y})=\big( a\mapsto  0, b\mapsto 0,c\mapsto  e^{y\bbi}\bbi\bbj\big).
$$
Translation in $SU(2)^3$ identifies these with the $ su(2)$-valued 1-cochains
 $$z_x=\big( a\mapsto 0, b\mapsto -\bbi,c\mapsto 0\big), z_y=\big( a\mapsto  0, b\mapsto 0,c\mapsto  -\bbi\big).$$
 
 Since $\hat m$ is a local diffeomorphism away from $(\pi\ZZ)^2$ \cite{HHK1}, and 
 $\tfrac{\partial}{\partial x},\tfrac{\partial}{\partial y}$ span $T_{(x,y)}\RR^2$,
 the cohomology classes $[z_x],~[z_y]$ span 
 $$\widehat H^1(S^2-4D^2; su(2)_{\ad \hat m(x,y)})=T_{[\hat m(x,y)]}(\chi (S^2,4)).$$
 
For any $(x,y)\in \RR^2\setminus(\pi\ZZ)^2$, the adjoint action 
$\ad \hat m(x,y):\pi_1(S^2-4D^2)\to GL(su(2))$ reduces as the direct sum 
 $$\ad \hat m(x,y)=\ad \hat m(x,y)_1\oplus\ad \hat m(x,y)_2 :\pi_1(S^2-4D^2)\to  GL(\RR\bbi)\times GL(\CC\bbj)
 $$
and hence
$$\widehat H^1(S^2-4D^2; su(2)_{\ad \hat m(x,y)})
=\widehat H^1(S^2-4D^2; \RR\bbi_{{\ad \hat m(x,y)}_1})\oplus
\widehat H^1(S^2-4D^2; \CC\bbj_{{\ad \hat m(x,y)}_2}).
$$
Since $\RR\bbi$ and $\CC\bbj$ are orthogonal, the symplectic form $\widehat \omega_{S^2,4}$ splits orthogonally 
$$\widehat \omega_{S^2,4}=\widehat \omega_{S^2,4}^1\oplus \widehat \omega_{S^2,4}^2.$$

The cocyles $z_x$ and $z_y$ lie in the first summand, and hence they span the first summand and the second summand is zero (these two facts can also be easily calculated directly). Hence
$\widehat \omega_{S^2,4}=\widehat \omega_{S^2,4}^1.$

The representation on the first summand independent of $(x,y)$: indeed $a,b,$ and $c$ (and hence also $d$) act by $-1$ for all $x,y$. The cocycles $z_x, z_y$ are   independent of $x,y$, and hence
$$
 \hat m^*(\widehat \omega_{S^2,4}) |_{(x,y)}(\tfrac{\partial}{\partial x} , \tfrac{\partial}{\partial y}) =\widehat \omega^1_{S^2,4}(z_x,z_y)
=c dx\wedge dy (\tfrac{\partial}{\partial x} , \tfrac{\partial}{\partial y})$$
for a non-zero constant $c$ (since $z_x,z_y$ span).

\medskip

The half-Dehn twists along the disks $H_1$ and $H_2$ illustrated in Figure \ref{halffig}   induce automorphisms
of $\pi_1(S^2\setminus 4, s_0)$, for a base point chosen outside the supports of $H_1$ and $H_2$, as indicated in the figure. These automorphisms are given by
$$H_1: a\mapsto a, b\mapsto b, c\mapsto d=\bar c\bar b\bar a$$
$$H_2: a\mapsto a, b\mapsto cd\bar c=\bar b \bar a\bar c, c\mapsto c$$
and induce homeomorphisms $H_1^*,H_2^*:\chi(S^2, 4)\to \chi(S^2,4)$.
Naturality of cup products shows that  $H_1^*$, $H_2^*$ preserve the symplectic form $\widehat \omega_{(S^2,4)}$.

Next:
\begin{align*}\label{eq4.7}H_1^*(m(x,y))&=\big(a\mapsto m(x,y)(a)=\bbj, b\mapsto m(x,y)(b)= e^{x\bbi}, c\mapsto m(x,y)(\bar c\bar b\bar a)=e^{(y-x)\bbi}\bbj\big)\\
&=m(x, y-x)
\end{align*}
and similarly
$$H_2^*(m(x,y)) =\big(a\mapsto \bbj, b\mapsto m(x,y)(\bar b \bar a \bar c)=e^{(x+y)\bbi}\bbj , c\mapsto e^{y\bbi})=m(x+y,y).
$$ 

It follows that the subgroup of Homeo$(\chi(S^2,4))$ generated by $H_1^*$ and $H_2^*$ pulls back, via the homeomorphism $\hat \rho:\mathbb{P}\to \chi(T)$, to the subgroup of $PSL(2,\ZZ)$ generated by 
\begin{equation*} 
\hat\rho^*(H_1^*)=\begin{pmatrix} 1&-1\\ 0&1 \end{pmatrix},~
\hat\rho^*(H_2^*)=\begin{pmatrix} 1&0\\1&1 \end{pmatrix}.
\end{equation*}
These two matrices generate $PSL(2,\ZZ)$ (\cite{MR607504}), finishing the proof.
 \end{proof}

Mapping classes   of $(S^2,4)$  permute the four punctures. The subgroup 
of the mapping class group of $(S^2,4)$ which fixes the point labelled $a$ in Figure \ref{halffig}
can be shown to act on $\chi(S^2,4)$; this subgroup is isomorphic to $PSL(2,\ZZ)$ (see, e.g., \cite{MR2850125}), generated by these half twists.

\color{black}

\section{Perturbations}
The holonomy perturbation process is easily understood, as well as motivated,  in the 
language of Weinstein composition of Lagrangian immersions.

\subsection{Composition}
Given any two (set) maps 
$$\alpha:A\to M \text{ and }\beta=\beta_M\times \beta_N:B\to M\times N,$$ 
define the {\em composition} $(A\times_M B,\beta^\alpha_N)$ by
\begin{equation}\label{eq4.1} A\times_M B:=\{ (a,b)\in A \times B ~|~\alpha(a)=\beta_M(b)\}=(\alpha\times \beta_M)^{-1}(\Delta_M)
\end{equation}
and
\begin{equation}
\beta^\alpha_N:A\times_M B\to N, ~ \beta^\alpha_N(a,b):=\beta_N(b).
\end{equation}
\subsubsection{Composition in character varieties}

Recall that if $\rho:\Gamma\to G$ is a homomorphism, 
$$Stab(\rho)=\{g\in G~|~g\rho(\gamma)g^{-1}=\rho(\gamma) \text{ for all } \gamma\in \Gamma\}.$$
The proof of the following simple lemma  can be found in \cite[Lemma 4.2]{HHK1}.
 
\begin{lemma}\label{cvsvk} Suppose that $\Gamma_0,\Gamma_1$, $H$ are groups,  
$h_0:H\to \Gamma_0$, $h_1:H\to \Gamma_1$ homomorphisms. Set $\Gamma=\Gamma_0*_H \Gamma_1$,  the pushout   along $h_0,h_1$. Then there is a surjection
$$\chi(\Gamma)\to \chi(\Gamma_0)\times_{\chi(H)}\chi(\Gamma_1)$$
with fiber over $([\rho_0],[\rho_1])$ the double coset
$$Stab(\rho_0)\backslash Stab(\rho_0|_H)/Stab(\rho_1).$$
\end{lemma}

The fibers $Stab(\rho_0)\backslash Stab(\rho_0|_H)/Stab(\rho_1)$
are called  {\em gluing parameters}.

\begin{proposition} \label{fiber} If $Z=Z_0\cup_{\Sigma}Z_1$ is a decomposition  of a compact 3-manifold along a closed separating surface $\Sigma$,  with $\pi_1(\Sigma)\to \pi_1(Y_0)$ surjective,   then  $\chi(Z)\to \chi(Z_0)\times_{\chi(\Sigma)}\chi(Z_1)$  is a homeomorphism, in fact, an algebraic isomorphism.
\end{proposition}
\begin{proof}
Choose $([\rho_0],[\rho_1])\in\chi(Z_0)\times_{\chi(\Sigma)}\chi(Z_1)$.  Since $\pi_1(\Sigma)\to \pi_1(Y_0)$ is 
surjective, $\rho_0$ and $\rho_0|_{\pi_1(\Sigma)}$ have the same image, and hence equal stabilizers.    Therefore $Stab(\rho_0)\backslash Stab(\rho_0|_\Sigma)/Stab(\rho_1)$ is a single point, and the proof follows from  Lemma \ref{cvsvk}.
\end{proof}
\begin{corollary}\label{handlebody}
 If $Z=Z_0\cup_\Sigma Z_1$ with $Z_0$ a handlebody, then
 $$\chi(Z)=\chi(Z_0)\times_{\chi(\Sigma)}\chi(Z_1).$$
\end{corollary}

\subsubsection{Composition in the Weinstein category} The Weinstein category
\cite{Weinstein} is, roughly speaking,  a category with objects symplectic manifolds and morphisms Lagrangian immersions.  Composition 
 is not always defined, however.  
The following criterion ensures that a composition of  Lagrangian immersions is defined.

\begin{lemma}\label{wkl}  \cite[§4.1]{MR657442}, \cite[lemma~2.0.5]{BW} Suppose that $M,N$ are symplectic manifolds, $$\alpha:A\to M \text{ and }\beta=\beta_M\times \beta_N:B\to M^-\times N$$ are Lagrangian immersions (with $M^-$ obtained from $M$ by reversing the sign of the symplectic form).
If $\alpha\times \beta_M$ is transverse to the diagonal $\Delta_M\in M\times M$, then $A\times_M B$ is a smooth manifold and $\beta^\alpha_N: A\times_M B\to N$ is a Lagrangian immersion.
\end{lemma}

When the transversality assumption in Lemma \ref{wkl} holds, one says {\em the composition $$(A\times_M B,\beta^\alpha_N)$$ of $(A,\alpha)$ and $(B,\beta)$  is defined and immersed}.

 \medskip

 In cases where the transversality assumption in Lemma \ref{wkl} does not hold, a differential topological approach to remedying the situation would be to deform either or both of the immersions $\alpha, \beta$.
 In order to retain the symplectic properties, one would typically deform them by Hamiltonian flows.  In the context in this article, where character varieties correspond to flat moduli spaces in the gauge theoretic framework, we seek deformations in Lemma \ref{wkl}   that correspond to {\em holonomy perturbations} in the gauge theory context;  we describe these  in the next section.

 \subsection{SU(2) and holonomy perturbations}\label{su2ahp}
We return to the pair $(Y,C)$, determining $S=V\cup F$ as above, so $S=\partial Y$ and $V=nbd(C)$. 

 Suppose that $e:D^2\times S^1\hookrightarrow Int(Y)$ is an embedding of  a  solid torus.   Denote its image by $Y_0$,  the closure of the complement by $Y_1$,  and the separating torus by $T$, so that
 $Y=Y_0\cup_T Y_1.$  Corollary \ref{handlebody} shows that 
 $$\chi(Y,C)=\chi(Y_0)\times_{\chi(T)}\chi(Y_1,C),$$ 
 so that $\chi(Y,C)\to \chi(F,\partial F)$ is exhibited as the composition of 
 $$\alpha: \chi(Y_0)\to \chi(T) \text{  and }
 \beta: \chi(Y_1,C) \to \chi(T)\times \chi(F,\partial F).$$

For a deformation $\alpha_\pi:\chi(Y_0)\to \chi(T)$ of $\alpha$,  or more generally a family of functions $\alpha_\pi, ~\pi \in \mathcal{U}$ with  $ \mathcal{U}$   a manifold (for example, a small open interval), we can view the deformed composition 
\begin{equation}\label{pertCV} \chi_{\pi}(Y,C):=\chi(Y_0)\circ_\pi \chi(Y_1,C), ~\beta^{\alpha_\pi }:\chi(Y_0)\circ_\pi  \chi(Y_1,C)\to \chi(F,\partial F)\end{equation} 
as a   {\em perturbed character variety with perturbation data  $\pi$}.

Restrict to the case when $G=SU(2)$ and $J$ is the conjugacy class of imaginary unit quaternions.  Then $\chi(Y_0)$ is simply an arc $[e^{is}],0\leq s \leq \pi$, and the immersion to $\chi(\partial Y_0)$ is $\alpha: [e^{is}]\mapsto [e^{is},1]$ in the pillowcase, described in  (\ref{eq4.5})   and (\ref{eq4.6}).  This map descends from the  smooth, $\ZZ_2$ equivariant map $e^{is} \to (e^{is},1)$ from $S^1 $ to $T$.    

Let $f:\RR\to \RR$ be a smooth, odd, $2\pi$ periodic function.  Then $f$ determines a $\ZZ_2$ equivariant Hamiltonian deformation $e^{is}\mapsto (e^{is},e^{if(s)})$, inducing the deformation  of $\alpha$ given by 
 \begin{equation}\label{perts}
  \alpha_f:[0,\pi]\to\chi(S^1\times S^1),  ~ \alpha_f(s)=[e^{\bbi s},e^{\bbi f(s)}].
  \end{equation}

\medskip 

The definition extends easily to the setting of   a finite disjoint collection of embeddings of solid tori $e=\{e_i\}_{i=1}^k$ of disjointly embedded solid tori in $Y$ and a corresponding collection of smooth, odd, periodic functions $f_1, \dots, f_k$ as above.  Denote by $\pi$ this set of perturbation data $\{(e_i, f_i)\}_{i=1}^k$.  
 Letting $Y_0$ denote the union of the solid tori, Equation (\ref{pertCV}) defines a way to deform the character variety.  
 
Notice that it suffices to think of $e$ as framed link in $Y$, since isotopic embeddings yield equal  perturbed character varieties.  
 
\begin{definition} \label{top desc hol pert} 
Let $\mathcal{V}$ be the  vector space of smooth, odd, $2\pi$ periodic functions $\RR\to \RR$.  Fix $(Y,C)$ as above, where  $C$ may or may not be nonempty).
Given perturbation data $$\pi=(e,f)=(\sqcup_{i=1}^ke_i:D^2\times S^1\subset Y, f=(f_i)\in \mathcal{V}^k),$$
define $\chi_\pi(Y,C)$ to be the resulting perturbed (traceless) character variety.
\end{definition}

In light of  Equation  (\ref{perts}), $\chi_\pi(Y,C)$ has the following explicit description.

\begin{proposition} Let $(Y,C)$ be a compact oriented 3-manifold with a collection of curves $C$ in its boundary.   Let $\pi=(e,f)$ be a choice of perturbation data, and define $Y_0\subset Y$ to  be the disjoint union of  solid tori  $\sqcup_i e_i(D^2\times S^1)$.  Finally define $\lambda_i,\mu_i\in \pi_1(Y\setminus Y_0)$  to be the loops $e_i(S^1\times \{1\})$ and  $e_i(  \{1\}\times \partial D^2)$, connected to the base point in some way.

Then 
$$
 \chi_\pi(Y,C)=\{ \rho\in \chi(Y\setminus Y_0,C)~|~ \text{ if $\rho(\lambda_i)= e^{\bbi s}$, then $\rho(\mu_i)=e^{\bbi f(s)}$} \}.
$$
 
\end{proposition}

\begin{theorem}[\cite{Taubes, Herald}] 
 Given a set of perturbation data $\pi=(e,f)$ as in Definition \ref{top desc hol pert}, there is a holonomy perturbation $h_{(e,f)}$ of the flatness equation on $SU(2)$ connections for which the perturbed flat moduli space is identified with the perturbed character variety as described above.\end{theorem}
  
 \begin{remark} More flexible   holonomy perturbations can be defined using a solid handlebody, rather than disjoint  solid tori (see, for example, \cite{Flo}). 
 In light of Corollary \ref{handlebody}, there is a similar composition interpretation of the perturbed character variety, with $Y_0$ the   handlebody. But 
 in   this setting,  an   explicit description of the perturbed character  variety   and the counterpart to the restriction map 
 in Equation (\ref{perts}), are not known to the authors.
\end{remark}

\medskip

We can now prove the following theorem, which  is equivalent to   Theorem \ref{thm1}.
\begin{theorem}
  \label{thm1A}
Suppose $A\in \chi_\pi(Y,C)$ is a regular point.   Then $A$ admits a neighborhood $U$ so that 
$r|_U:U\to  \chi(F, \partial F)^*$ is a Lagrangian embedding.
\end{theorem}

\begin{proof}
Apply Lemma \ref{wkl}.
\end{proof}

 \subsection{Dependence on perturbations}\label{ssec:depend_pert} 
  In \cite{Herald}, the second author showed that (when $\cL$ is empty) different
  holonomy perturbations in general yield Legendrian cobordant immersions. Unfortunately this usually does not guarantee that they have isomorphic Floer homology. We now outline how one can address this point using Wehrheim and Woodward's quilt theory \cite{WW} (a rigorous formulation of the Weinstein category \cite{MR2662868}),   as well as its extension to the immersed case as developed by Bottman-Wehrheim \cite{BW}.

 \medskip

The reader will find a discussion of holonomy perturbations in cylinders $(S,p)\times[0,1]$, with $p$ a finite set of points in a surface $S$, in \cite{HK}. In particular, Theorem 6.3 of that article states the following. Take  perturbation data $\pi$ with framed perturbation curve obtained by pushing a simple closed curve in $S\setminus p$ into the interior of $S\times I$. Then define  the 1-parameter family of holonomy perturbations $s\pi, s\in [0,\ep)$.

The restriction
$$\chi_{s\pi}((S,p)\times I)\to \chi(S_0,p)\times \chi(S_1,p)$$
 can be identified with the family of graphs of a Hamiltonian isotopy
 of $\chi(S)$  known as the {\em Goldman twist flow} associated to the simple closed curve \cite{Goldman}.   This implies that the holonomy perturbation  process 
 can be viewed as a combination of a decomposition induced by cutting a 3 manifold along a separating torus $T$, followed by a perturbation as in Equation (\ref{pertCV}), with $\alpha_{s\pi}:\chi(S^1\times D^2)\to \chi(T)$ the composition of the unperturbed inclusion $\alpha_0$ followed by a small time flow of the Hamiltonian Goldman twist flow associated to a curve in this torus.

\medskip

 We formalize this in the following way.    
  Call two   Lagrangian immersions $$\iota_0\colon L_0 \looparrowright  M,~\iota_1\colon L_1 \looparrowright  M$$  \emph{secretly Hamiltonian isotopic}  if they can be expressed as compositions with some $\beta:\Lambda \looparrowright  X\times M$:
  \begin{align*}
 i_0: L_0\looparrowright M &=  \beta^{j_0}_M\colon L_0'\times_X\Lambda\looparrowright M, \\
  i_1: L_1\looparrowright M &= \beta^{j_1}_M\colon L_1'\times_X\Lambda\looparrowright M.
    \end{align*}
in such a way that $j_0,j_1:L_0', L_1' \looparrowright X$ are Hamiltonian isotopic in $X$. 
It follows from the discussion in Section~\ref{su2ahp} that different choices of holonomy perturbations induce secretly Hamiltonian isotopic immersed Lagrangians.

Being secretly Hamiltonian isotopic does not necessarily imply being Hamiltonian isotopic. Indeed, $L_0$ and $L_1$ need not even be diffeomorphic. Moreover, in the absence of extra hypotheses (embedded composition, monotonicity, exactness, etc.), the conclusion of the Wehrheim-Woodward composition theorem \cite{WW} need not hold. In particular, given a third Lagrangian $L_2$,  $HF(L_0, L_2)$ and $HF(L_1, L_2)$ need not  be isomorphic, even if they are both well-defined.

However, provided:
\begin{itemize}

\item all Lagrangian immersions and symplectic manifolds satisfy suitable assumptions so to be able to define Lagrangian (quilted) Floer homology,

\item all Lagrangian immersions come equipped with suitable bounding cochains, in a way consistent with composition.

\end{itemize}
Then it would follow from the Bottman-Wehrheim conjecture \cite[Sec.~4.4]{BW} (see also a similar statement in \cite{Fuk_Ymaps}) that the secretly Hamiltonian isotopic Lagrangian immersions $L_0$ and $L_1$, when paired with any test Lagrangian $L_2\subset M$, produce isomorphic Floer homology groups $HF(L_0, L_2)\simeq HF(L_1, L_2)$ (since these would respectively correspond to the quilted Floer homology groups $HF(L_0',  \Lambda, L_2)$ and $HF(L_1',  \Lambda, L_2)$, which are isomorphic).

In that sense  secretly Hamiltonian isotopic Lagrangian immersions $L_0$ and $L_1$ can be thought as being equivalent.  In particular, the problem of  dependence of   Floer theory on the choice of holonomy perturbation
is seen as a special case of the general problem of dependence of quilted Floer homology on composition.

\newcommand*{\arxivPreprint}[1]{ArXiv preprint \href{http://arxiv.org/abs/#1}{#1}}
\newcommand*{\arxiv}[1]{ArXiv:\ \href{http://arxiv.org/abs/#1}{#1}}
\newcommand*{\web}[1]{\ \href{#1}{#1}}
\bibliographystyle{alpha} 
\bibliography{PD}

\end{document}